\documentclass[11pt]{article}
\usepackage{amsfonts, amssymb, amsmath, amsthm, xcolor, graphicx, pstricks, pstricks-add}
\begin{document}

\theoremstyle{plain}
\newtheorem{thm}{Theorem}[section]
\newtheorem{cor}[thm]{Corollary}
\newtheorem{con}[thm]{Conjecture}
\newtheorem{cla}[thm]{Claim}
\newtheorem{lm}[thm]{Lemma}
\newtheorem{prop}[thm]{Proposition}
\newtheorem{example}[thm]{Example}

\theoremstyle{definition}
\newtheorem{dfn}[thm]{Definition}
\newtheorem{alg}[thm]{Algorithm}
\newtheorem{prob}[thm]{Problem}
\newtheorem{rem}[thm]{Remark}

\renewcommand{\baselinestretch}{1.1}

\title{\bf On AZ-style identity}
\author{
Cheng Yeaw Ku
\thanks{ Department of Mathematics, National University of Singapore, Singapore 117543. E-mail: matkcy@nus.edu.sg} \and Kok Bin Wong \thanks{
Institute of Mathematical Sciences, University of Malaya, 50603 Kuala Lumpur, Malaysia. E-mail:
kbwong@um.edu.my.} } \maketitle

\begin{abstract}\noindent
The AZ identity is a generalization of the LYM-inequality. In this paper, we will give a generalization of the AZ identity.
\end{abstract}

\bigskip\noindent
{\sc keywords:} LYM-inequality, AZ-identity\newline

\section{Introduction}
 Let $[n]=\{1,2,\dots, n\}$, $\Omega_n$ be the family of all subsets of $[n]$, and $\varnothing$ be the empty set. 
 Let $\varnothing\neq \mathcal F\subseteq \Omega_n$. If $A\nsubseteq B$ for all $A,B\in\mathcal F$ with $A\neq B$, then $\mathcal F$ is called a \emph{Sperner family} or \emph{antichain}. For any antichain $\mathcal F$, the following inequality holds:
 \begin{equation}
 \sum_{X\in\mathcal F} \frac{1}{\binom{n}{\vert X\vert}}\leq 1.\tag{1}
 \end{equation}
The inequality (1) is called the LYM-inequality (Lubell, Yamamoto, Meshalkin) (see \cite[Chapter 13]{Bollo}).
Many generalizations of the LYM-inequality have been obtained (see \cite{Bey, CG, EFK, ES}). In particular, Ahlswede and Zhang \cite{AZ} discovered an identity (see equation (2)) in which the LYM-inequality is a consequence of it. 

Let $\mathbf G_n$ be the family of all $\mathcal F$ such that $\varnothing\neq \mathcal F\subseteq \Omega_n$. For every $\mathcal F\in\mathbf G_n$, the set
\begin{equation}
D_n(\mathcal F)=\{ Y\subseteq [n]\ :\ Y\subseteq F\ \textnormal{for some $F\in\mathcal F$}\},\notag
\end{equation}
is called the \emph{downset}, while the set
\begin{equation}
U_n(\mathcal F)=\{ Y\subseteq [n]\ :\ Y\supseteq F\ \textnormal{for some $F\in\mathcal F$}\},\notag
\end{equation}
is called the \emph{upset}. For each $X\subseteq [n]$, we set
\begin{equation}
Z_{\mathcal F}(X)=\begin{cases}
\varnothing & \ \textnormal{if $X\notin U_n(\mathcal F)$},\\
\bigcap_{F\in\mathcal F, F\subseteq X} F &\ \textnormal{otherwise}.
\end{cases}\notag
\end{equation}

\begin{thm}\label{thm_AZ_identity}\textnormal{\cite{AZ}} For any $\mathcal F\in \mathbf G_n$ with $\varnothing\notin \mathcal F$, 
\begin{equation}
\sum_{X\in U_n(\mathcal F)} \frac{\vert Z_{\mathcal F}(X)\vert}{\vert X\vert\binom{n}{\vert X\vert}}=1.\tag{2}
\end{equation}\qed
\end{thm}

Equation (2) is called the \emph{AZ-identity}. Note that when $\mathcal F$ is an antichain, $Z_{\mathcal F}(F)=F$ for all $F\in\mathcal F$. So equation (2) becomes
\begin{equation}
\sum_{F\in\mathcal F} \frac{1}{\binom{n}{\vert F\vert}}+\sum_{X\in U_n(\mathcal F)\setminus \mathcal F} \frac{\vert Z_{\mathcal F}(X)\vert}{\vert X\vert\binom{n}{\vert X\vert}}=1,\notag
\end{equation}
and as the second term on the left is non-negative, we obtain inequality (1).

Later, Ahlswede and Cai discovered an identity  for two set systems.

\begin{thm}\label{thm_AZ_identity2}\textnormal{\cite{AC1}} Let $\mathcal A=\{A_1,A_2,\dots, A_q\}$ and $\mathcal B=\{B_1,B_2,\dots, B_q\}$ be elements in $\mathbf G_n$. Suppose that $A_i\neq\varnothing$ for all $i$, and $A_j\subseteq B_k$ if and only if $j=k$. Then
\begin{equation}
\sum_{i=1}^q \frac{1}{\binom{n-\vert B_i\vert+\vert A_i\vert}{\vert A_i\vert}}+\sum_{X\in U_n(\mathcal A)\setminus D_n(\mathcal B)} \frac{\vert Z_{\mathcal A}(X)\vert}{\vert X\vert\binom{n}{\vert X\vert}}=1.\tag{3}
\end{equation}\qed
\end{thm}

Ahlswede and Cai \cite{AC2} also discovered AZ type of identities of several other posets. For the duality of equations (2) and (3), we refer the readers to \cite{Day, HS}.

Recently, Thu discovered the following generalizations of equations (2) and (3).

\begin{thm}\label{thm_AZ_identity1_general}\textnormal{\cite{Thu3}} Let $m$ be an integer, and $\mathcal F\in \mathbf G_n$ with $\varnothing\notin \mathcal F$. If $\vert F\vert +m> 0$ for all $F\in\mathcal F$, then
\begin{equation}
\sum_{X\in U_n(\mathcal F)} \frac{\vert Z_{\mathcal F}(X)\vert+m}{(\vert X\vert+m)\binom{n+m}{\vert X\vert+m}}=1.\tag{4}
\end{equation}\qed
\end{thm}

\begin{thm}\label{thm_AZ_identity2_general}\textnormal{\cite{Thu3}} Let $m$ be an integer, and $\mathcal A=\{A_1,A_2,\dots, A_q\}$ and $\mathcal B=\{B_1,B_2,\dots, B_q\}$ be elements in $\mathbf G_n$. Suppose that $A_i\neq\varnothing$ for all $i$, and $A_j\subseteq B_k$ if and only if $j=k$. If $\vert A\vert +m> 0$ for all $A\in\mathcal A$, then 
\begin{equation}
\sum_{i=1}^q \frac{1}{\binom{n+m-\vert B_i\vert+\vert A_i\vert}{\vert A_i\vert+m}}+\sum_{X\in U_n(\mathcal A)\setminus D_n(\mathcal B)} \frac{\vert Z_{\mathcal A}(X)\vert+m}{(\vert X\vert+m)\binom{n+m}{\vert X\vert+m}}=1.\tag{5}
\end{equation}\qed
\end{thm}

In this paper, we will give generalizations of equations (4) and (5) (see Theorem \ref{thm_main1} and Theorem \ref{thm_main2}).

\section{Main theorems}

Let us denote the set of real numbers by $\mathbb R$ and the set of natural numbers by $\mathbb N$. Let $a,m\in\mathbb R$ and $n\in\mathbb N$. Suppose that $ak+m\neq 0$ for $k=l,l+1,\dots, n$. We set 
\begin{equation}
g_{a,m}(n,l)=\frac{(n-l)!a^{n-l}}{\prod_{k=l}^n (ak+m)}.\notag
\end{equation}

\begin{lm}\label{lm_first_link} Suppose $l<n$. If $ak+m\neq 0$ for $k=l,l+1,\dots, n$, then
\begin{equation}
g_{a,m}(n,l)+g_{a,m}(n,l+1)=g_{a,m}(n-1,l).\notag
\end{equation}
\end{lm}

\begin{proof} Note that
\begin{align}
g_{a,m}(n,l)+g_{a,m}(n,l+1) &=\frac{(n-l)!a^{n-l}}{\prod_{k=l}^n (ak+m)}+\frac{(n-l-1)!a^{n-l-1}}{\prod_{k=l+1}^n (ak+m)}\notag\\
&=\frac{(n-l)!a^{n-l}+(al+m)(n-l-1)!a^{n-l-1}}{\prod_{k=l}^n (ak+m)}\notag\\
&=\frac{n(n-l-1)!a^{n-l}+m(n-l-1)!a^{n-l-1}}{\prod_{k=l}^n (ak+m)}\notag\\
&=\frac{(n-l-1)!a^{n-l-1}}{\prod_{k=l}^{n-1} (ak+m)}\notag\\
&=g_{a,m}(n-1,l).\notag
\end{align}
\end{proof}

The following lemma can be verified easily.

\begin{lm}\label{lm_special_case} Suppose that $ak+m\neq 0$ for $k=l,l+1,\dots, n$.
\begin{itemize}
\item[\textnormal{(a)}] If $a=1$ and $m$ is an integer, then
\begin{equation}
g_{1,m}(n,l)=\frac{1}{(l+m)\binom{n+m}{l+m}}.\notag
\end{equation}
\item[\textnormal{(b)}] If $a=1$ and $m=0$, then
\begin{equation}
g_{1,0}(n,l)=\frac{1}{(l)\binom{n}{l}}.\notag
\end{equation}\qed
\end{itemize}
\end{lm}

We shall need the following lemma (see equation (3) of \cite{Thu1}, or Lemma 2 of \cite{Day}).

\begin{lm}\label{lm_set_thu} Let $\varnothing\notin \mathcal A\in\mathbf G_n$ and $\varnothing\notin \mathcal B\in\mathbf G_n$. Set 
\begin{equation}
\mathcal A\vee\mathcal B=\{A\cup B\ :\ A\in\mathcal A, B\in\mathcal B\}.\notag
\end{equation}
Then for each $\varnothing\neq X\subseteq [n]$, 
\begin{equation}
\left\vert Z_{\mathcal A\cup\mathcal B}(X)\right\vert=\left\vert Z_{\mathcal A}(X)\right\vert+\left\vert Z_{\mathcal B}(X)\right\vert-\left\vert Z_{\mathcal A\vee\mathcal B}(X)\right\vert.\notag
\end{equation}\qed
\end{lm}

\begin{thm}\label{thm_main1} Let $a,m\in\mathbb R$ and $n\in\mathbb N$.  Let $\varnothing\notin \mathcal A\in\mathbf G_n$. Suppose that $ak+m\neq 0$ for all $\min_{A\in\mathcal A} \vert A\vert\leq k\leq n$. Then
\begin{equation}
\sum_{X\in U_n(\mathcal A)} \left(a\vert Z_{\mathcal A}(X)\vert+m\right)g_{a,m}(n,\vert X\vert)=1.\tag{6}
\end{equation}
\end{thm}

\begin{proof} {\bf Case 1}. Suppose $\mathcal A=\{A\}$. We may assume that $A=\{1,2,\dots ,r\}$. 

Note that if  $r=n$, then $U_n(\mathcal A)=\{ A\}$, $Z_{\mathcal A}(A)=A$, and $\sum_{X\in U_n(\mathcal A)} \left(a\vert Z_{\mathcal A}(X)\vert+m\right)g_{a,m}(n,\vert X\vert)=(an+m)g_{a,m}(n,n)=1$. Suppose $r<n$. Note that
$U_n(\mathcal A)=U_{n-1}(\mathcal A)\cup\{ X\cup \{n\}\ :\ X\in U_{n-1}(\mathcal A)\}$, and $Z_{\mathcal A}(X)=Z_{\mathcal A}(X\cup \{n\})$. Therefore by Lemma \ref{lm_first_link},
\begin{align}
\sum_{X\in U_n(\mathcal A)} &\left(a\vert Z_{\mathcal A}(X)\vert+m\right)g_{a,m}(n,\vert X\vert)\notag\\
&=\sum_{X\in U_{n-1}(\mathcal A)} \left(a\vert Z_{\mathcal A}(X)\vert+m\right)g_{a,m}(n,\vert X\vert)\notag\\
&\hskip 1cm+\sum_{X\in U_{n-1}(\mathcal A)} \left(a\vert Z_{\mathcal A}(X\cup\{n\})\vert+m\right)g_{a,m}(n,\vert X\vert+1)\notag\\
&=\sum_{X\in U_{n-1}(\mathcal A)} \left(a\vert Z_{\mathcal A}(X)\vert+m\right)\left(g_{a,m}(n,\vert X\vert)+g_{a,m}(n,\vert X\vert+1)\right)\notag\\
&=\sum_{X\in U_{n-1}(\mathcal A)} \left(a\vert Z_{\mathcal A}(X)\vert+m\right)g_{a,m}(n-1,\vert X\vert).\notag
\end{align}
If $r=n-1$, then $U_{n-1}(\mathcal A)=\{ A\}$, $Z_{\mathcal A}(A)=A$, and $\sum_{X\in U_{n-1}(\mathcal A)} \left(a\vert Z_{\mathcal A}(X)\vert+m\right)g_{a,m}(n-1,\vert X\vert)=(a(n-1)+m)g_{a,m}(n-1,n-1)=1$. So the theorem holds. Suppose $r<n-1$. Again by Lemma \ref{lm_first_link}, 
\begin{align}
\sum_{X\in U_{n-1}(\mathcal A)} &\left(a\vert Z_{\mathcal A}(X)\vert+m\right)g_{a,m}(n-1,\vert X\vert)\notag\\
&=\sum_{X\in U_{n-2}(\mathcal A)} \left(a\vert Z_{\mathcal A}(X)\vert+m\right)g_{a,m}(n-2,\vert X\vert).\notag
\end{align}
By continuing this way, we see that
\begin{align}
\sum_{X\in U_{n}(\mathcal A)} &\left(a\vert Z_{\mathcal A}(X)\vert+m\right)g_{a,m}(n,\vert X\vert)\notag\\
&=\sum_{X\in U_{r}(\mathcal A)} \left(a\vert Z_{\mathcal A}(X)\vert+m\right)g_{a,m}(r,\vert X\vert)\notag\\
&=(ar+m)g_{a,m}(r,r)\notag\\
&=1.\notag
\end{align}

\noindent
{\bf Case 2}. Suppose $\mathcal A=\{A_1,\dots, A_q\}$, $q\geq 2$. Assume that the theorem holds for all $q'$ with $1\leq q'<q$. Let $\mathcal B=\{A_1,\dots, A_{q-1}\}$ and $\mathcal C=\{A_q\}$. Then $\mathcal B\vee\mathcal C=\{A_1\cup A_q,\dots , A_{q-1}\cup A_q\}$, $U_n(\mathcal A)=U_n(\mathcal B)\cup U_n(\mathcal C)$ and $U_n(\mathcal B\vee\mathcal C)=U_n(\mathcal B)\cap U_n(\mathcal C)$. By Lemma \ref{lm_set_thu}, 
\begin{equation}
\left\vert Z_{\mathcal A}(X)\right\vert=\left\vert Z_{\mathcal B}(X)\right\vert+\left\vert Z_{\mathcal C}(X)\right\vert-\left\vert Z_{\mathcal B\vee\mathcal C}(X)\right\vert.\notag
\end{equation}
So if $X\in U_{n}(\mathcal B)\setminus U_n(\mathcal C)$, then $\left\vert Z_{\mathcal A}(X)\right\vert=\left\vert Z_{\mathcal B}(X)\right\vert$, if $X\in U_{n}(\mathcal C)\setminus U_n(\mathcal B)$, then $\left\vert Z_{\mathcal A}(X)\right\vert=\left\vert Z_{\mathcal C}(X)\right\vert$, and if $X\in U_{n}(\mathcal B)\cap  U_n(\mathcal C)$, then $\left\vert Z_{\mathcal A}(X)\right\vert=\left\vert Z_{\mathcal B}(X)\right\vert+\left\vert Z_{\mathcal C}(X)\right\vert-\left\vert Z_{\mathcal B\vee\mathcal C}(X)\right\vert$.

Therefore
\begin{align}
\sum_{X\in U_{n}(\mathcal A)} &\left(a\vert Z_{\mathcal A}(X)\vert+m\right)g_{a,m}(n,\vert X\vert)\notag\\
&=\sum_{X\in U_{n}(\mathcal B)\setminus U_n(\mathcal C)} \left(a\vert Z_{\mathcal B}(X)\vert+m\right)g_{a,m}(n,\vert X\vert)\notag\\
&\hskip 0.5cm +\sum_{X\in U_{n}(\mathcal C)\setminus U_n(\mathcal B)} \left(a\vert Z_{\mathcal C}(X)\vert+m\right)g_{a,m}(n,\vert X\vert)\notag\\
&\hskip 1cm +\sum_{X\in U_n(\mathcal B\vee\mathcal C)} \left(a\left(\left\vert Z_{\mathcal B}(X)\right\vert+\left\vert Z_{\mathcal C}(X)\right\vert-\left\vert Z_{\mathcal B\vee\mathcal C}(X)\right\vert\right)+m\right)g_{a,m}(n,\vert X\vert)\notag\\
&=\sum_{X\in U_{n}(\mathcal B)} \left(a\vert Z_{\mathcal B}(X)\vert+m\right)g_{a,m}(n,\vert X\vert)\notag\\
&\hskip 0.5cm +\sum_{X\in U_{n}(\mathcal C)} \left(a\vert Z_{\mathcal C}(X)\vert+m\right)g_{a,m}(n,\vert X\vert)\notag\\
&\hskip 1cm -\sum_{X\in U_n(\mathcal B\vee\mathcal C)} \left(a\left\vert Z_{\mathcal B\vee\mathcal C}(X)\right\vert+m\right)g_{a,m}(n,\vert X\vert),\notag
\end{align}
and by induction, 
\begin{equation}
\sum_{X\in U_{n}(\mathcal A)} \left(a\vert Z_{\mathcal A}(X)\vert+m\right)g_{a,m}(n,\vert X\vert)=1+1-1=1.\notag
\end{equation}
\end{proof}

Note that by  Lemma \ref{lm_special_case}, equations (2) and (4) are consequence of Theorem \ref{thm_main1}.

We shall need the following lemma (see Lemma 4 of \cite{Thu3}).

\begin{lm}\label{lm_Thu2} Let $\mathcal A_1,\mathcal A_2,\mathcal B_1,\mathcal B_2\in\mathbf G_n$ and $\varnothing\notin \mathcal A_1\cup \mathcal A_2\cup \mathcal B_1\cup\mathcal B_2$. Suppose that $U_n(\mathcal A_1)\cap D_n(\mathcal B_2)=\varnothing=U_n(\mathcal A_2)\cap D_n(\mathcal B_1)$. Let $\mathcal A=\mathcal A_1\cup\mathcal A_2$ and $\mathcal B=\mathcal B_1\cup\mathcal B_2$. If $F$ is a non-zero function defined on $U_n(\mathcal A)$, then
\begin{align}
\sum_{X\in U_{n}(\mathcal A)\setminus D_n(\mathcal B)} \frac{a\vert Z_{\mathcal A}(X)\vert+m}{F(X)} &=\sum_{X\in U_{n}(\mathcal A_1)\setminus D_n(\mathcal B_1)} \frac{a\vert Z_{\mathcal A_1}(X)\vert+m}{F(X)}\notag\\
&\hskip 0.5cm +\sum_{X\in U_{n}(\mathcal A_2)\setminus D_n(\mathcal B_2)} \frac{a\vert Z_{\mathcal A_2}(X)\vert+m}{F(X)}\notag\\
&\hskip 1cm -\sum_{X\in U_{n}(\mathcal A_1\vee \mathcal A_1)} \frac{a\vert Z_{\mathcal A_1\vee \mathcal A_2}(X)\vert+m}{F(X)}.\notag
\end{align}\qed
\end{lm}

In fact Lemma \ref{lm_Thu2} can be proved easily by noting that 
\begin{align}
U_{n}(\mathcal A)\setminus D_n(\mathcal B)&=(U_{n}(\mathcal A_1)\setminus (D_n(\mathcal B_1)\cup U_n(\mathcal A_2)))\notag\\
&\hskip 1cm \cup (U_{n}(\mathcal A_2)\setminus (D_n(\mathcal B_2)\cup U_n(\mathcal A_1))) \cup U_{n}(\mathcal A_1\vee \mathcal A_2),\notag
\end{align}
and by applying Lemma \ref{lm_set_thu}.

\begin{lm}\label{lm_pre_main2}  Let $a,m\in\mathbb R$ and $n\in\mathbb N$. Let $A,B$ be non-empty subsets of $[n]$. If $A\subseteq B$, and $ak+ m\neq 0$ for all $\vert A\vert\leq k\leq n$, then
\begin{equation}
\sum_{A\subseteq X\subseteq B} g_{a,m}(n,\vert X\vert)=g_{a,m}(n-\vert B\vert+\vert A\vert,\vert A\vert).\notag
\end{equation}
\end{lm}

\begin{proof}  We may assume that $A=\{1,2,\dots, r_1\}$ and $B=\{1,2,\dots, r_1,r_1+1,\dots, r_2\}$. We shall prove by induction on $p=r_2-r_1$.

Suppose $p=0$, i.e., $A=B$. Then
\begin{align}
\sum_{A\subseteq X\subseteq B} g_{a,m}(n,\vert X\vert)&=g_{a,m}(n,\vert A\vert).\notag
\end{align}
Suppose $p>1$. Assume that the lemma holds for $p'<p$. 

Note that $A\subsetneq B$ and $r_2\notin A$.  Set $B'=B\setminus \{r_2\}$. Then $A\subseteq B'$, and by Lemma \ref{lm_first_link},
\begin{align}
\sum_{A\subseteq X\subseteq B} g_{a,m}(n,\vert X\vert)&=\sum_{A\subseteq X\subseteq B'}g_{a,m}(n,\vert X\vert)+\sum_{A\subseteq X\subseteq B'}g_{a,m}(n,\vert X\cup \{r_{2}\}\vert)\notag\\
&=\sum_{A\subseteq X\subseteq B'} \left(g_{a,m}(n,\vert X\vert)+g_{a,m}(n,\vert X\vert+1)\right)\notag\\
&=\sum_{A\subseteq X\subseteq B'} g_{a,m}(n-1,\vert X\vert).\notag
\end{align}
By induction $\sum_{A\subseteq X\subseteq B'} g_{a,m}(n-1,\vert X\vert)=g_{a,m}(n-1-\vert B'\vert+\vert A\vert,\vert A\vert)=g_{a,m}(n-\vert B\vert+\vert A\vert,\vert A\vert)$. Hence $\sum_{A\subseteq X\subseteq B} g_{a,m}(n,\vert X\vert)=g_{a,m}(n-\vert B\vert+\vert A\vert,\vert A\vert)$.
\end{proof}

\begin{thm}\label{thm_main2} Let $a,m\in\mathbb R$ and $n\in\mathbb N$. Let $\mathcal A=\{A_1,A_2,\dots, A_q\}$ and $\mathcal B=\{B_1,B_2,\dots, B_q\}$ be elements in $\mathbf G_n$. Suppose that $A_i\neq\varnothing$ for all $i$, and $A_j\subseteq B_k$ if and only if $j=k$. If $ak+m\neq 0$ for all $\min_{A\in\mathcal A} \vert A\vert\leq k\leq n$, then
\begin{equation}
\sum_{i=1}^q (a\vert A_i\vert+m)g_{a,m}(n-\vert B_i\vert+\vert A_i\vert,\vert A_i\vert)+\sum_{X\in U_n(\mathcal A)\setminus D_n(\mathcal B)} \left(a\vert Z_{\mathcal A}(X)\vert+m\right)g_{a,m}(n,\vert X\vert)=1.\tag{7}
\end{equation}
\end{thm}

\begin{proof} {\bf Case 1}. Suppose $q=1$. Then $\mathcal A=\{A_1\}$, $\mathcal B=\{B_1\}$, $\varnothing\neq A_1\subseteq B_1$, and $a\vert A_1\vert +m\neq 0$. Furthermore if $X\in U_n(\mathcal A)$, then $Z_{\mathcal A}(X)=A_1$. By Theorem \ref{thm_main1}, 
\begin{equation}
\sum_{X\in U_n(\mathcal A)\cap D_n(\mathcal B)} \left(a\vert Z_{\mathcal A}(X)\vert+m\right)g_{a,m}(n,\vert X\vert)+\sum_{X\in U_n(\mathcal A)\setminus D_n(\mathcal B)} \left(a\vert Z_{\mathcal A}(X)\vert+m\right)g_{a,m}(n,\vert X\vert)=1.\notag
\end{equation}
 Note that by Lemma \ref{lm_pre_main2}
 \begin{align}
 \sum_{X\in U_n(\mathcal A)\cap D_n(\mathcal B)} \left(a\vert Z_{\mathcal A}(X)\vert+m\right)g_{a,m}(n,\vert X\vert) &=(a\vert A_1\vert+m)\sum_{A\subseteq X\subseteq B} g_{a,m}(n,\vert X\vert)\notag\\
 &=(a\vert A_1\vert+m)g_{a,m}(n-\vert B_1\vert+\vert A_1\vert,\vert A_1\vert).\notag
 \end{align}
 Hence the theorem holds.

{\bf Case 2}. Suppose $q>1$. Assume that the theorem holds for all $q'$ with $1\leq q'<q$. Let 
\begin{align}
\mathcal A_1 &=\{A_1,\dots, A_{q-1}\}, & \mathcal A_2 &=\{A_q\},\notag\\
\mathcal B_1 &=\{B_1,\dots, B_{q-1}\}, & \mathcal B_2 &=\{B_q\}.\notag
\end{align}
Note that $U_n(\mathcal A_1)\cap D_n(\mathcal B_2)=\varnothing=U_n(\mathcal A_2)\cap D_n(\mathcal B_1)$. By Lemma \ref{lm_Thu2} and induction,
\begin{align}
\sum_{X\in U_{n}(\mathcal A)\setminus D_n(\mathcal B)} \frac{a\vert Z_{\mathcal A}(X)\vert+m}{F(X)} &=\left(1-\sum_{i=1}^{q-1} (a\vert A_i\vert+m)g_{a,m}(n-\vert B_i\vert+\vert A_i\vert,\vert A_i\vert) \right)\notag\\
&\hskip 0.5cm +\left(1- (a\vert A_q\vert+m)g_{a,m}(n-\vert B_q\vert+\vert A_q\vert,\vert A_q\vert) \right)\notag\\
&\hskip 1cm -\sum_{X\in U_{n}(\mathcal A_1\vee \mathcal A_1)} \left(a\vert Z_{\mathcal A_1\vee \mathcal A_2}(X)\vert+m\right)g_{a,m}(n,\vert X\vert).\notag
\end{align}
Note that by Theorem \ref{thm_main1}, the $\sum_{X\in U_{n}(\mathcal A_1\vee \mathcal A_1)} \left(a\vert Z_{\mathcal A_1\vee \mathcal A_2}(X)\vert+m\right)g_{a,m}(n,\vert X\vert)=1$. Hence the theorem holds.
\end{proof}

Note that by  Lemma \ref{lm_special_case}, equations (3) and (5) are consequence of Theorem \ref{thm_main2}.

\end{document}